\documentclass[11pt,reqno]{amsart}

\usepackage{color}
\usepackage[colorlinks=true,allcolors=blue,backref=page]{hyperref}
\usepackage{amsmath, amssymb, amsthm}
\usepackage{mathrsfs}
\usepackage{mathtools}
\usepackage[noabbrev,capitalize,nameinlink]{cleveref}
\crefname{equation}{}{}
\usepackage{fullpage}
\usepackage[noadjust]{cite}
\usepackage{graphics}
\usepackage{pifont}
\usepackage{tikz}
\usepackage{bbm}
\usepackage[T1]{fontenc}

\usetikzlibrary{arrows.meta}

\usepackage{environ}
\usepackage{framed}
\usepackage{url}
\usepackage[linesnumbered,ruled,vlined]{algorithm2e}
\usepackage[noend]{algpseudocode}
\usepackage[labelfont=bf]{caption}
\usepackage{cite}
\usepackage{framed}
\usepackage[framemethod=tikz]{mdframed}
\usepackage{appendix}
\usepackage{graphicx}
\usepackage[textsize=tiny]{todonotes}
\usepackage{tcolorbox}
\usepackage{enumerate}
\usepackage{verbatim}
\allowdisplaybreaks[1]

\apptocmd{\sloppy}{\hbadness 10000\relax}{}{} 

\crefname{algocf}{Algorithm}{Algorithms}

\crefname{equation}{}{} 
\crefname{conjecture}{Conjecture}{Conjectures} 
\AtBeginEnvironment{appendices}{\crefalias{section}{appendix}} 

\usepackage[color,final]{showkeys} 

\colorlet{refkey}{orange!20}
\colorlet{labelkey}{blue!30}

\crefname{algocf}{Algorithm}{Algorithms}

\numberwithin{equation}{section}
\newtheorem{theorem}{Theorem}[section]
\newtheorem{proposition}[theorem]{Proposition}
\newtheorem{lemma}[theorem]{Lemma}

\crefname{claim}{Claim}{Claims}

\newtheorem*{question*}{Question}

\theoremstyle{definition}
\newtheorem{definition}[theorem]{Definition}

\newtheorem*{definition*}{Definition}

\theoremstyle{remark}
\newtheorem*{remark}{Remark}

\newcommand{\snorm}[1]{\lVert#1\rVert}

\newcommand{\bs}{\boldsymbol}
\newcommand{\mb}{\mathbb}
\newcommand{\mbf}{\mathbf}
\newcommand{\mbm}{\mathbbm}
\newcommand{\mc}{\mathcal}

\newcommand{\mr}{\mathrm}

\newcommand{\on}{\operatorname}
\newcommand{\wh}{\widehat}
\newcommand{\wt}{\widetilde}

\allowdisplaybreaks

\title{Random symmetric matrices: rank distribution and irreducibility of the characteristic polynomial}

\author[Ferber]{Asaf Ferber}

\address{Department of Mathematics, University of California, Irvine.}

\email{asaff@uci.edu}

\author[Jain]{Vishesh Jain}
\address{Department of Statistics, Stanford University, Stanford, CA 94305, USA}
\email{visheshj@stanford.edu}

\author[Sah]{Ashwin Sah}
\author[Sawhney]{Mehtaab Sawhney}

\address{Department of Mathematics, Massachusetts Institute of Technology, Cambridge, MA 02139, USA}
\email{\{asah,msawhney\}@mit.edu}

\thanks{Ferber was supported in part by NSF grants DMS-1954395 and DMS-1953799. Sah and Sawhney were supported by NSF Graduate Research Fellowship Program DGE-1745302.}

\begin{document}

\begin{abstract}
Conditional on the extended Riemann hypothesis, we show that with high probability, the characteristic polynomial of a random symmetric $\{\pm 1\}$-matrix is irreducible. This addresses a question raised by Eberhard in recent work. The main innovation in our work is establishing sharp estimates regarding the rank distribution of symmetric random $\{\pm 1\}$-matrices over $\mb{F}_p$ for primes $2 < p \leq \exp(O(n^{1/4}))$. Previously, such estimates were available only for $p = o(n^{1/8})$. At the heart of our proof is a way to combine multiple inverse Littlewood--Offord-type results to control the contribution to singularity-type events of vectors in $\mb{F}_p^{n}$ with anticoncentration at least $1/p + \Omega(1/p^2)$. Previously, inverse Littlewood--Offord-type results only allowed control over vectors with anticoncentration at least $C/p$ for some large constant $C > 1$.   
\end{abstract}

\maketitle

\section{Introduction}\label{sec:introduction}
The irreducibility of random polynomials has attracted much interest in recent years. A well-known conjecture of Odlyzko and Poonen \cite{OP93} is that the random polynomial $P(x) = x^{d} + b_{d-1}x^{d-1} + \dots + b_{1}x + b_{0}$, where $b_{0} = 1$ and $b_1,\dots, b_{d-1}$ are i.i.d.~$\on{Ber}(1/2)$ random variables (i.e.~each $b_i$ is independently $0$ or $1$ with probability $1/2$ each), is irreducible in $\mb{Z}[x]$ with probability $1-o_d(1)$. This was established in a more general form by Breuillard and Varj\'u \cite{BV19} under the Riemann Hypothesis for a family of Dedekind zeta functions. A version of this conjecture, where $b_0,\dots, b_{d-1}$ are distributed uniformly in $\{1,\dots, L\}$ for $L$ divisible by at least $4$ distinct primes, was established (unconditionally) by Bary-Soroker and Kozma \cite{BK20}. In recent work, Bary-Soroker, Koukoulopoulos, and Kozma \cite{bary2020irreducibility} showed that the result continues to hold for $\{1,\dots, L\}$ for $L\geq 35$. We refer the reader to \cite{BV19,BK20,bary2020irreducibility} for more precise and general versions of the aforementioned results. 

Another popular model of random polynomials is the characteristic polynomial of a random matrix. It was conjectured by Babai in the 1970s (and again, by Vu and Wood in 2009) that for an $n\times n$ matrix $N_n$ whose entries are i.i.d.~Rademacher random variables (i.e.~$\pm 1$ with probability $1/2$ each), the characteristic polynomial $\wh{\varphi}(t) = \det(tI_n-N_n)$ is irreducible with probability $1-o_n(1)$. This was confirmed, under the extended Riemann Hypothesis, in recent work of Eberhard \cite{Ebe20}, building on \cite{BV19} and ideas from the non-asymptotic theory of random matrices. It is perhaps even more natural to consider the (real-rooted) characteristic polynomial $\varphi(t) = \det(tI_n - M_n)$, where $M_n$ is an $n\times n$ symmetric matrix whose entries on and above the diagonal are i.i.d.~Rademacher random variables. In \cite{Ebe20}, Eberhard asked if $\varphi(t)$ is irreducible with probability $1-o_n(1)$. We answer this question in the affirmative under the extended Riemann hypothesis. 


\begin{theorem}\label{thm:main}
Assume the extended Riemann Hypothesis (ERH) (i.e., the Riemann Hypothesis for Dedekind zeta functions for all number fields). Then there is an absolute constant $c>0$ such that characteristic polynomial $\varphi(t) = \det(tI_n - M_n)$ of an $n\times n$ random symmetric Rademacher matrix $M_n$ is irreducible with probability at least $1-2\exp(-cn^{1/4})$.
\end{theorem}
\begin{remark}
The proof in this paper can easily be extended to handle the class of $\alpha$-balanced distributions considered in \cite{Ebe20} with straightforward modifications; we leave the details to the interested reader. 
\end{remark}

It was noted in \cite{Ebe20} that given the techniques in \cite{BV19,Ebe20}, \cref{thm:main} (with the weaker probability bound $1-o_n(1)$) can be deduced from the following universality statement: the probability that $tI_n - M_{n}$ is invertible over $\mb{F}_p$, for $p = n^{\Omega(1)}$, is essentially the same as for an $n\times n$ symmetric matrix whose entries on and above the diagonal are sampled from the uniform distribution on $\mb{F}_p$. Despite the intensive efforts to study the singularity probability of symmetric Rademacher matrices (\cite{CTV06, Ngu11, Ver14, FJ19, CMMM20, JSS20a, CJMS20} and especially the recent breakthrough \cite{CJMS21} which confirms the long-standing conjecture that the singularity probability of symmetric Rademacher matrices is exponentially small), a result of this precision has remained elusive. While for $p = o(n^{1/8})$, such a result is known due to work of Maples \cite{Map13}, the bound on $p$ is too restrictive to imply \cref{thm:main} (even with the weaker probability $1-o_n(1)$). The main challenge in addressing the regime $p = \omega(n^{1/2})$ is that one must consider the arithmetic structure of vectors which are orthogonal to random subspaces of small co-dimension. However, inverse Littlewood--Offord type theorems (cf.~\cite{Tao06, RV08, FJLS2018}), which have been designed to study only arithmetically structured vectors, fail to apply to vectors in $\mb{F}_{p}$ with anticoncentration at least $C/p$ for some large constant $C > 1$. While consideration of vectors with anticoncentration at most $C/p$ is inessential for the less precise results mentioned above, here we must provide an appropriate structural result for vectors with anticoncentration at least $1/p + \Omega(1/p^2)$, say. This is accomplished in the key \cref{prop:structure-many}. Finally, we note that an upper bound on the singularity probability of the form $1/p + O(1/p^2)$ can be deduced for $p=o(n^{1/2})$ from estimates on the expected size of the kernel over $\mb{F}_p$ due to \cite{Fer20}, but for similar reasons to those mentioned above these estimates do not appear to extend to $p$ which is larger than a small polynomial.

An obvious generalization of studying the probability of singularity of $M_n$ over $\mb{F}_p$ is studying the rank distribution of $M_n$ over $\mb{F}_p$. For symmetric Rademacher matrices, the only prior work we are aware of is the aforementioned work of Maples \cite{Map13}, which effectively requires $p = o(n^{1/8})$. Results for unrestricted $p$ are available under the very strong assumption that the independent entries of $M_n$ are uniformly \cite{FG15} or nearly-uniformly \cite{koenig2020rank} distributed over $\mb{F}_p$. 

The main innovation of our work is the following result regarding the rank distribution of symmetric Rademacher matrices (and diagonal perturbations) over $\mb{F}_p$ for all $2 < p \leq \exp(\eta n^{1/4})$: 
\begin{theorem}\label{thm:sharp-probability}
There exists $\eta > 0$ so that for any $2 < p\le\exp(\eta n^{1/4})$ and for any $\lambda\in\mb{F}_p$, the $n\times n$ symmetric Rademacher matrix $M_n$ satisfies
\[\mb{P}[\on{rank}_{\mb{F}_p}(M_n-\lambda I_n) = n-k] = \frac{\prod_{i=0}^{\infty}(1-p^{-(2i+1)})}{\prod_{i=1}^{k}(p^i-1)} + O(\exp(-\eta n/\log p)).\]
\end{theorem}
\begin{remark}
The proof can be extended routinely to the class of $\alpha$-balanced distributions; however, for the sake of brevity, we have restricted our attention to the Rademacher distribution. 
\end{remark}
\begin{remark}
\cref{thm:sharp-probability} is the natural symmetric analog of the results in \cite{Map10,LMN19}. As mentioned above, a version was known for $p$ sufficiently small (with weaker error terms) due to Maples \cite{Map13}.
\end{remark}

\subsection{Organization}\label{sub:organization}
The remainder of this paper is organized as follows. In \cref{sec:structure}, we prove our key structural result (\cref{prop:structure-many}) for vectors which are orthogonal to many rows of $M_n$. In \cref{sec:result}, we use this, along with arguments in \cite{Map13,koenig2020rank} to deduce \cref{thm:sharp-probability}. \cref{sec:nt} contains the deduction of \cref{thm:main} from the $k=1$ case of \cref{thm:sharp-probability}, following the arguments in \cite{Ebe20}. Finally, \cref{sec:app-proof-structure-1} contains the proof of a `crude' structure theorem (which appears in \cite{FJ19}, but with worse parameters) for the reader's convenience.  



\section{Structure Theorem for Almost-Kernel Vectors}\label{sec:structure}

We begin by showing the easy fact that, except with exponentially small probability, no sparse vector has sparse image under $M$. 

\begin{definition}
Let $0 \leq r \leq n$ be a parameter. We say that $\bs{v}\in\mb{F}_p^n$ is an $r$-kernel vector of a matrix $M\in\mb{F}_p^{n\times n}$ if $M\bs{v}$ is $r$-sparse. 
\end{definition}
In particular, $0$-kernel vectors correspond to the usual right-kernel of $M$. 

\begin{lemma}\label{lem:not-sparse}
Let $p \ge 3$. With probability at least $1 - \exp(-n/6)$, the symmetric random matrix $M_n - \lambda I_n$ has no $n/(16\log p)$-sparse $n/4$-kernel vectors in $\mb{F}_p^n$.
\end{lemma}
\begin{proof}
Fix a non-zero $\bs{v} \in \mb{F}_p^n$ with $v_1 \neq 0$. We begin by computing the probability that $M\bs{v}$ is $n/4$-sparse. We denote the rows of $M$ by $R_1,\dots, R_n$ and reveal the rows from bottom-to-top. Since the first entry $R_i$ is independent $R_{i+1},\dots, R_n$, since $v_1 \neq 0$, and since $p\geq 3$, it follows that 
\[\max_{R_{i+1},\dots, R_{n}}\mb{P}[R_i\cdot \bs{v} = 0 \mid R_{i+1},\dots, R_{n}] \leq \frac{1}{2}.\]
Therefore, the probability that $M\bs{v}$ is $n/4$-sparse is at most
\[\frac{1}{2^{n-n/4}}\cdot \binom{n}{n/4}.\]
Finally, taking the union bound over the at most $$\exp(nH(1/(16\log p))p^{n/(8\log p)}$$ choices of $n/(16\log{p})$-sparse vectors in $\mb{F}_{p}^{n}$ gives the desired conclusion.
\end{proof}

We recall the definition of the atom probability of a vector $\bs{v} \in \mb{F}_{p}^{n}$ with respect to Rademacher random variables. 

\begin{definition}
The atom probability of a vector $\bs{v} \in \mb{F}_{p}^{n}$ is defined as
\[\rho_{\mb{F}_p}(\bs{v}) = \max_{r \in \mb{F}_p}\mb{P}[\xi_1 v_1 + \dots + \xi_n v_n = r],\]
where $\xi_1,\dots, \xi_n$ are i.i.d.~Rademacher random variables. 
\end{definition}

We will need the following `crude' structure theorem for $(n-n/\log p)$-kernel vectors, which follows from a more careful version of the argument in \cite{FJ19}. For the reader's convenience, we include details in \cref{sec:app-proof-structure-1}. 

\begin{proposition}\label{prop:structure-1}
Suppose that $2 < p\le\exp(c_{\ref{prop:structure-1}}n^{1/4})$ and fix $\lambda\in\mb{F}_p$. With probability at least $1-\exp(-n/8)$, every vector $\bs{v} \in \mb{F}_p^{n}$ which is orthogonal to at least $n-n/\log p$ rows of the random symmetric matrix $M_n-\lambda I_n$ over $\mb{F}_p$ satisfies the following two properties:
\begin{itemize}
\item $|\on{supp}(\bs{v})| \geq n/(16\log{p})$, and  
\item There exists $S\subseteq\on{supp}(\bs{v})$ with $|S| \in [\sqrt{n\log n},n^{3/4}]$ such that
\[\rho_{\mb{F}_p}(\bs{v}|_S)\le\frac{C_{\ref{prop:structure-1}}}{p}.\]
\end{itemize}
\end{proposition}

Since 
\[\rho_{\mb{F}_p}(\bs{v}) \leq \rho_{\mb{F}_p}(\bs{v}|_S)\]
for any $S\subseteq [n]$, the crude structure theorem shows that any $(n/\log p)$-kernel vector of $M_n - \lambda I_n$ has atom probability at most $\frac{C_{\ref{prop:structure-1}}}{p}$. While this result is optimal up to the constant $C_{\ref{prop:structure-1}}$, it is unfortunately insufficient for our application. The next proposition, which is one of the main innovations of this paper, allows us to show that any $(n/\log p)$-kernel vector of $M_{n}-\lambda I_n$ has many disjoint chunks with atom-probability approximately $1/p$.

\begin{proposition}\label{prop:structure-many}
Suppose that $2 < p\le\exp(c_{\ref{prop:structure-many}}n^{1/4})$ and fix $\lambda\in\mb{F}_p$. Let $m = C_{\ref{prop:structure-many}}\log p$. With probability at least $1-\exp(-n/9)$, every vector $\bs{v}$ orthogonal to at least $n-n/\log p$ rows of the random symmetric matrix $M_n-\lambda I_n$ over $\mb{F}_p$ has the following property, which we denote by \emph{(}$\dagger$\emph{)}: given any set $T\subseteq[n]$ of size $n/4$, there at least $n/(2m)$ disjoint sets $S$ of size $m$ in $[n]\setminus T$ such that
\[\on{disc}_{\mb{F}_p}(\bs{v}|_S) = \sup_{x\in\mb{F}_p}\bigg|\mb{P}_\xi\bigg[\sum_{i\in S}\xi_iv_i = x\bigg]-\frac{1}{p}\bigg|\le\frac{C_{\ref{prop:structure-many}}}{p^2}.\]
\end{proposition}
\begin{proof}
This is trivial for $p\le C_{\ref{prop:structure-many}}^{1/2}$, so assume the opposite. Let $r = n/\log p$.

First, by \cref{lem:not-sparse,prop:structure-1}, we may assume that $\bs{v}$ has some $S\subseteq\on{supp}(\bs{v})$ with $|S|\in [\sqrt{n\log n},n^{3/4}]$ such that
\[\rho_{\mb{F}_p}(\bs{v}|_S)\le\frac{C_{\ref{prop:structure-1}}}{p}.\]
For some fixed vector with this property, a straightforward argument (cf.~proof of \cref{prop:structure-1}) shows that the probability that it is orthogonal to some set of $n-r$ rows of $M_n$ is at most
\begin{equation}
\label{eqn:crude}
\binom{n}{r}\cdot\bigg(\frac{C_{\ref{prop:structure-1}}}{p}\bigg)^{n-n^{3/4}-r}\le C^np^{-n}.
\end{equation}

We will use a union bound argument to establish \cref{prop:structure-many}. Given \cref{eqn:crude}, the key is to show that the set of vectors $\bs{v}$ which fail to satisfy ($\dagger$) is sufficiently sparse in $\mb{F}_p^n$. To this end, consider some $T$ of size $n/4$ and for each such $T$, consider a fixed (but otherwise arbitrary) partition of $[n]\setminus T$ into $S_1,\ldots,S_{3n/(4m)}$ of size $m$ (up to rounding). There are at most $2^n$ ways to choose $T$ and at most $2^n$ ways to choose which of these sets satisfy
\[\on{disc}_{\mb{F}_p}(\bs{v}|_{S_i})\le C_{\ref{prop:structure-many}}/p^2.\]
If $\bs{v}$ violates ($\dagger$), then at least a third of these indices are failures. Thus, we see that the number of vectors which violate ($\dagger$) is at most
\begin{equation}
\label{eqn:union-size}
2^n\cdot2^n\cdot(p^{3n/4})\cdot|T|^{n/(4m)},
\end{equation}
where $T$ is the set of $\mb{F}_p$-vectors of size $m$ with
\[\on{disc}_{\mb{F}_p}(\bs{v}|_{S_i}) > C_{\ref{prop:structure-many}}/p^2.\]
Fix some sufficiently large absolute constant $C' > 0$ (depending on $C$ in \cref{eqn:crude}). We claim that if $C_{\ref{prop:structure-many}}\ge 1$ chosen large enough, then $|T|\le(p/C')^m$ (recall that $p > C_{\ref{prop:structure-many}}^{1/2}$). Note that this claim, together with \cref{eqn:crude} and \cref{eqn:union-size} completes the proof of \cref{prop:structure-many}.  

We now prove the claim. Note that if $\bs{b} = (b_1,\ldots,b_m)\in T\subseteq\mb{F}_p^m$ for i.i.d.~Rademacher random variables $\xi_1,\dots, \xi_n$ we have that
\begin{align*}
\frac{1}{p^2} < \on{disc}_{\mb{F}_p}(\bs{b}) &= \sup_{x\in\mb{F}_p}\bigg|\mb{P}_\xi[b_1\xi_1+\cdots+b_m\xi_m=x]-\frac{1}{p}\bigg|\\
&= \sup_{x\in\mb{F}_p}\bigg|\frac{1}{p}\sum_{\ell\in\mb{F}_p^\times}\exp(2\pi i(\ell x/p))\prod_{j=1}^m\cos(2\pi(\ell b_j/p))\bigg|\\
&\le\frac{1}{p}\sum_{\ell\in\mb{F}_p^\times}\prod_{j=1}^m|\cos(2\pi(\ell b_j/p))|\\
&\le\frac{1}{p}\sum_{\ell\in\mb{F}_p^\times}\exp\bigg(-\sum_{j=1}^m\snorm{2\ell b_j/p}_{\mb{R}/\mb{Z}}^2\bigg)\\
&\le\max_{\ell\in\mb{F}_p^\times}\exp\bigg(-\sum_{j=1}^m\snorm{\ell b_j/p}_{\mb{R}/\mb{Z}}^2\bigg).
\end{align*}
In particular, there is some $\ell\in\mb{F}_p^\times$ with
\[\sum_{j=1}^m\snorm{\ell b_j/p}_{\mb{R}/\mb{Z}}^2\le 2\log p.\]
To count the number of such vectors, note that there are at most $p$ choices of $\ell$. Moreover, since coordinate-wise multiplication by $\ell$ is a bijection from $\mb{F}_{p}^{m}$ onto itself, it follows that
\[|T| \leq p\cdot\#\{\bs{b}\in \mb{F}_p^m : \sum_{j=1}^{m}\|b_j/p\|^{2}_{\mb{R}/\mb{Z}} \leq 2\log{p}\}.\]
The second term in the product is a count of lattice points in a ball. A standard volumetric argument shows that there are at most $\on{vol}(B_2^m(R+\sqrt{m}))$ integer lattice points in an $m$-dimensional ball of radius $R$. Hence, we see that
\[|T|\le p\cdot\bigg(1+\frac{p\sqrt{2\log p}}{\sqrt{m}}\bigg)^m\le 2^m\cdot\Big(1+p\sqrt{2C_{\ref{prop:structure-many}}^{-1}}\Big)^m\le\Big(8C_{\ref{prop:structure-many}}^{-1/2}p\Big)^m\]
since $m = C_{\ref{prop:structure-many}}\log{p}$ and $p > C_{\ref{prop:structure-many}}^{1/2}$. Choosing $C_{\ref{prop:structure-many}}$ sufficiently large completes the proof. 
\end{proof}

\section{Proof of \texorpdfstring{\cref{thm:sharp-probability}}{Theorem 1.2}}\label{sec:result}
We are now in position to deduce \cref{thm:sharp-probability}; the high-level structure of the argument is as in \cite{Map13, koenig2020rank}. However, we make explicit a number of details which are implicit in \cite{Map13} as well as make a number of simplifications and changes to account for use of our structure theorem \cref{prop:structure-many}. 

We first recall a basic linear algebra fact about full-rank principal minors.
\begin{lemma}\label{lem:full-minor}
If $M\in\mb{F}^{n\times n}$ is a symmetric matrix of rank $r$, then there is an invertible $r\times r$ principal minor of $M$.
\end{lemma}

Fix $\lambda\in\mb{F}_p$. We consider the nested sequence of symmetric matrices
\[M_1-\lambda I_1\subseteq M_2-\lambda I_2\subseteq\cdots\subseteq M_n-\lambda I_n,\]
where $M_{i} - \lambda I_{i}$ is the top-left $i\times i$-submatrix of $M - \lambda I$ (hence, the nested sequence is obtained by iteratively adding symmetric ``reverse-L'' shapes of Rademacher random variables, with a shift by $\lambda$ on the diagonal element). For simplicity, let $A_t = M_t - \lambda I_t$. Define the events 
$$\mc{E}_{S,t} = \{A_t \text{ has no nonzero } t/(16\log p)\text{-sparse } t/4\text{-kernel vector}\}, \quad \text{and}$$
$$\mc{E}_{U,t} =  \{\text{every nonzero }(t/\log p)\text{-kernel vector }\bs{v} \text{ of }A_t \text{ satisfies property (}\dagger\text{)}\},$$ 
where recall that ($\dagger$) is the property that given any set $T\subseteq[t]$ of size $t/4$, there at least $t/(2m)$ disjoint sets $S$ of size $m = C_{\ref{prop:structure-many}}\log{p}$ in $[t]\setminus T$ such that
\[\on{disc}_{\mb{F}_p}(\bs{v}|_S) = \sup_{x\in\mb{F}_p}\bigg|\mb{P}_\xi\bigg[\sum_{i\in S}\xi_iv_i = x\bigg]-\frac{1}{p}\bigg|\le\frac{C_{\ref{prop:structure-many}}}{p^2}.\]
By \cref{lem:not-sparse,prop:structure-many}, we see that
\begin{equation}
\label{eqn:good-event}
\mb{P}[\mc{E}_{S,t}^c\vee\mc{E}_{U,t}^c]\le 2\exp(-t/9).
\end{equation}

\begin{lemma}\label{lem:walk-transition}
We have for $2 < p\le\exp(c_{\ref{lem:walk-transition}}n^{1/4})$, fixed $\lambda\in\mb{F}_p$, and for any $k\ge 0$ that
\begin{align*}
\mb{P}[\on{corank}A_{t+1} = k-1|A_{t} : \on{corank}A_t=k\wedge\mc{E}_{S,t}\wedge\mc{E}_{U,t}] &= 1 - p^{-k} + O(\exp(-\Omega(t))),\\
\mb{P}[\on{corank}A_{t+1} = k + 0|A_{t} : \on{corank}A_t=k\wedge\mc{E}_{S,t}\wedge\mc{E}_{U,t}] &= p^{-k}-p^{-k-1}+ O(\exp(-\Omega(t/\log p))),\\
\mb{P}[\on{corank}A_{t+1} = k+1|A_{t} : \on{corank}A_t=k\wedge\mc{E}_{S,t}\wedge\mc{E}_{U,t}] &= p^{-k-1} + O(\exp(-\Omega(t/\log p))).
\end{align*}
\end{lemma}
\begin{remark}
Note that, without the error terms inside $O(\cdot)$, the probabilities are exactly the same as for the uniform model (i.e.~the independent entries of $M_n$ are chosen uniformly from $\mb{F}_p$).
\end{remark}

Before proving \cref{lem:walk-transition}, let us show how it implies \cref{thm:sharp-probability}.

\begin{proof}[Proof of \cref{thm:sharp-probability}]
We consider the exposure process obtained by iteratively revealing $M_t$ for $1\le t\le n$ and considering the resulting corank of $A_t$. Starting from a random $A_{n/20}$, let $\tau$ denote the (random) first value of $t\geq n/20$ such that either (i) $A_{t} \in \mc{E}_{S,t}^{c} \cup \mc{E}_{U,t}^{c}$, or (ii) $\on{corank}(A_t) = 0$. We claim that, except with probability at most $O(\exp(-\Omega(n/\log{p})))$, we have $\tau \leq n/2$, and moreover, $A_{\tau}$ satisfies condition (ii) and not condition (i). 

To see this, note that by \cref{lem:not-sparse,prop:structure-many}, the probability that $A_{t} \in \mc{E}_{S,t}^{c} \cup \mc{E}_{U,t}^{c}$ is $O(\exp(-\Omega(n)))$ for any $t\geq n/20$. Moreover, for $A_{t} \in \mc{E}_{S,t} \wedge \mc{E}_{U,t}$ with $\on{corank}(A_{t}) = k \geq 1$, we see that $\mb{P}[\on{corank}A_{t+1} = k-1 \mid A_{t}] \geq 1-\frac{1}{3} + O(\exp(-\Omega(n))) \geq \frac{3}{5}$ for all $n$ sufficiently large, and similarly, $\mb{P}[\on{corank}A_{t+1} = k+1 \mid A_{t}] \leq \frac{1}{8}$ for all $n$ sufficiently large. Therefore, by a straightforward comparison argument, it follows that for all $n$ sufficiently large, the probability that $\tau \leq n/2$ and $A_{\tau}$ satisfies condition (ii) and not condition (i) is at most
\[O(\exp(-\Omega(n/\log{p}))) + q,\]
where $q = O(\exp(-\Omega(n)))$ is the probability that a biased random walk with steps $-1$ with probability $1/2$, $0$ with probability $1/4$, and $+1$ with probability $1/4$ and initial state $n/20$ does not hit $0$ in $n/2 - n/20$ steps. 

To summarize, we have shown that except with probability $O(\exp(-\Omega(n/\log{p})))$, there exists some $\tau \in [n/20, n/2]$ such that $A_{\tau} \in \mc{E}_{S,\tau} \wedge \mc{E}_{U,\tau}$ and $\on{corank}(A_{\tau}) = 0$. From this point onwards, outside an event of probability at most $O(\exp(-\Omega(n/\log{p})))$, it follows by the remark following \cref{lem:walk-transition} that we can couple the corank process $A_{\tau+1},\dots, A_{n}$ with the corank process $A'_{1},\dots, A'_{n-\tau}$ where $A'_{i}$ is the top-left $i\times i$ sub-matrix of $M'_{n-\tau} - \lambda$ and $M'_{n-\tau}$ is an $(n-\tau) \times (n-\tau)$ random symmetric matrix whose independent entries are chosen uniformly from $\mb{F}_p$. By \cite[Theorem~4.1]{FG15}, for $\tau \in [n/20, n/2]$ and for any $0 \leq k \leq n-\tau$,
\[\mb{P}[\on{corank}A'_{n-\tau} = k] =  \frac{\prod_{i=0}^{\infty}(1-p^{-(2i+1)})}{\prod_{i=1}^{k}(p^i-1)} + O(p^{-\Omega(n)}),\]
which completes the proof. \qedhere

\end{proof}

Finally, we prove \cref{lem:walk-transition}

\begin{proof}[Proof of \cref{lem:walk-transition}]
Note that $\on{rank} A_{t+1}-\on{rank} A_t\in\{0,1,2\}$ since rank is monotone and sub-additive and the rank of the ``reverse-L'' is at most $2$. Since the ambient dimension increases by $1$ at each step, it follows that the corank increases by one of the three values $\{0,\pm1\}$, so that it suffices to prove the first and third equalities. Write
\[A_{t+1} = \begin{bmatrix}A_t&\xi_t\\\xi_t^T&z\end{bmatrix}\]
where $\xi_t$ is an i.i.d.~Rademacher vector and $z+\lambda$ is an independent Rademacher.
Let $\xi$ be the vector $[\xi_t^T z]^T$.

\textbf{Step 1: Corank decrease via linear forms.}
The first equality is only nontrivial when $k\ge 1$, and in this case we need the rank to increase by $2$. Basic linear algebra (cf.~\cite[Lemma~2.4]{CTV06}) shows that this is equivalent to requiring $\xi_t$ to lie outside the span of the column space of $A_t$. Equivalently, $\xi_t$ should not be orthogonal to all kernel vectors of $A_t$ (which form a dimension $k$ subspace of $\mb{F}_p^t$). Let $\bs{v}_1,\ldots,\bs{v}_k$ form a basis of the kernel. We have
\begin{align*}
\left|\mb{P}_\xi[\bs{v}_j\cdot\xi_t = 0\text{ for all }j\in[k]]-p^{-k}\right| &=  \bigg|\frac{1}{p^k}\sum_{\bs{a}\in\mb{F}_p^k\setminus\{\mbf{0}\}}\mb{E}_\xi\exp\bigg(\frac{2\pi i}{p}(\xi_t\cdot(a_1\bs{v}_1+\cdots+a_k\bs{v}_k))\bigg)\bigg|\\
&\le\max_{\bs{a}\in\mb{F}_p^k\setminus\{\bs{0}\}}\bigg|\mb{E}_\xi\exp\bigg(\frac{2\pi i}{p}(\xi_t\cdot(a_1\bs{v}_1+\cdots+a_k\bs{v}_k))\bigg)\bigg|.
\end{align*}

Let $\bs{a} = (a_1,\dots, a_k) \in \mb{F}_p^{k} \setminus \{\bs{0}\}$ denote the element attaining the maximum. Since $\bs{v}_1,\dots, \bs{v}_k$ are linearly independent vectors in the kernel of $A_t$, $\bs{v} = a_1\bs{v}_1+\cdots+a_k\bs{v}_k$ is a nonzero kernel vector of $A_t$. Since we have conditioned on $\mc{E}_{S,t}\wedge \mc{E}_{U, t}$, $\bs{v}$ has at least $t/(16\log p)$ non-zero entries and satisfies property ($\dagger$). 

From $\mc{E}_{U,t}$ we see that $\bs{v}$ can be partitioned into at least $t/(2m)$ disjoint sets $S_1,\dots, S_{t/2m}$ of size $m$ with $\on{disc}_{\mb{F}_p}(\bs{v}|_{S_i})\le C_{\ref{prop:structure-many}}p^{-2}$ for all $i \in [t/2m]$. Hence,
\begin{align*}
    \mb{E}_{\xi}\exp\left(\frac{2\pi i}{p}\xi_t\cdot \bs{v}\right)
    &\leq \prod_{\ell=1}^{t/2m}\left|\mb{E}_{\xi|_{S_\ell}}\exp\left(\frac{2\pi i}{p} \xi|_{S_\ell}\cdot \bs{v}|_{S_\ell}\right)\right|\\
    &\leq \prod_{\ell=1}^{t/2m}\left|\sum_{j \in \mb{F}_p}\exp\left(\frac{2\pi i}{p}j\right)\cdot \left(\frac{1}{p} + \on{disc}_{\mb{F}_p}(\bs{v}|_{S_\ell})\right)\right|\\
    &\leq \prod_{\ell=1}^{t/2m}\left|\sum_{j \in \mb{F}_p}\on{disc}_{\mb{F}_p}(\bs{v}|_{S_\ell})\right|\\
    &\leq (C_{\ref{prop:structure-many}}p^{-1})^{t/(2m)} = \exp(-\Omega(t)),
\end{align*}
if $p > C_{\ref{prop:structure-many}}$. 

From $\mc{E}_{S,t}$ we see that $\bs{v}$ has support size at least $t/(16\log p)$, so that
\begin{align*}
     \mb{E}_{\xi}\exp\left(\frac{2\pi i}{p}\xi_t\cdot \bs{v}\right)
     &\leq \prod_{j \in \on{supp}(\bs{v})}\left|\mb{E}_{\xi_j}\exp\left(\frac{2\pi i}{p}\xi_j v_j\right)\right|\\
     &\leq (1-\Omega(1/p^2))^{t/16\log{p}}\\
     &= \exp(-\Omega(t/(p^2\log p))) = \exp(-\Omega(t)),
\end{align*}
if $p\le C_{\ref{prop:structure-many}}$. Therefore regardless of what $p > 2$ is, we have
\[\mb{P}_\xi[\bs{v}_j\cdot\xi_t = 0\text{ for all }j\in[k]] = p^{-k} + O(\exp(-\Omega(t))),\]
as desired to establish the first equality.

\textbf{Step 2: Corank increase via quadratic forms.}
Now we turn to the third equality. First note that it suffices to prove the claim when $k\le t/(64\log p)$ since if $k > t/(64\log p)$, the first equality already implies that the second and third probabilities are of size $p^{-k} + O(\exp(-\Omega(t))) = O(\exp(-\Omega(t)))$. 

By \cref{lem:full-minor}, $A_t$ has a principal minor of rank $(t-k)$. Therefore, without loss of generality, we may suppose that the top left $(t-k)\times(t-k)$ block, call it $B$, has full rank. Let $\phi$ be the restriction of $\xi$ to these coordinates. In order for the corank of $A_{t+1}$ to be larger than the corank of $A_{t}$, it must be the case that $\on{rank}(A_{t+1}) = \on{rank}(A_{t})$. This precisely corresponds to the event 
\[\{\exists \bs{w} \in \mb{F}_p^{t-k}: \xi_{t} = A_{t}\bs{w}\} \wedge \{\phi^TB^{-1}\phi = z\}.\]
Indeed, if $\xi_t$ is not in the column space then the rank of $A_{t+1}$ must increase and if $\phi^TB^{-1}\phi\neq -z$, then $B$ along with the new elements in the ``reverse-L'' will have rank $(t-k+1)$. On the other hand, if the above event holds, then it is easy to see that $\on{rank}(A_{t+1}) = \on{rank}(A_t)$.

As in the first case, let $\bs{v}_1,\dots, \bs{v}_k$ form a basis of the kernel of $A_{t}$. Then, 
\begin{align*}
\bigg|\mb{P}_\xi[\bs{v}_j\cdot\xi_t &= 0\text{ for all }j\in[k]\wedge\phi^TB^{-1}\phi = z] - p^{-k-1}\bigg|\\
&\le\sup_{\bs{a}\in\mb{F}_p^{k+1}\setminus\{\mbf{0}\}}\bigg|\mb{E}_\xi\exp\bigg(\frac{2\pi i}{p}(\xi_t\cdot(a_1\bs{v}_1+\cdots+a_k\bs{v}_k) + a_{k+1}(\phi^TB^{-1}\phi-z))\bigg)\bigg|.
\end{align*}
Note here that $\xi_t,\phi,z$ all depend on $\xi$. Let $\bs{a} = (a_1,\dots, a_{k+1}) \in \mb{F}_p^{k+1}\setminus \{\bs{0}\}$ denote the element attaining the maximum. If $a_{k+1} = 0$, we have a bound of $\exp(-\Omega(t))$ exactly as in the first case. It therefore suffices to consider $a_{k+1}\neq 0$. Let $\bs{v} = a_1\bs{v}_1+\cdots+a_k\bs{v}_k$.

To handle the quadratic term $\phi^T B^{-1} \phi$, we will use a decoupling trick, which in this context essentially goes back to \cite{CTV06}. Let $I\sqcup J = \{1,\ldots,t\}$ be the partition with $J = [t/(64\log p)]$, and let $\xi_I,\xi_J$ be the obvious restrictions. Let $\xi_J'$ be an independent resample of $\xi_J$. Let $\xi = \xi_I+\xi_J$ and $\xi' = \xi_I+\xi_J'$. Let $\phi' = \xi'\mid_{[t-k]}$. We have that
\begin{align*}
\bigg|\mb{E}_\xi&\exp\bigg(\frac{2\pi i}{p}(\xi_t\cdot \bs{v} + a_{k+1}(\phi^TB^{-1}\phi - z))\bigg)\bigg|^2\\
&= \mb{E}_{\xi_I,\xi_J,\xi_J'}\exp\bigg(\frac{2\pi i}{p}((\xi_J-\xi_J')\cdot\bs{v} + a_{k+1}(\phi^TB^{-1}\phi - (\phi')^TB^{-1}\phi'))\bigg)\\
&\leq \mb{E}_{\xi_J,\xi_J'}\left|\mb{E}_{\xi_I}\left[\exp\left(\frac{2\pi i}{p} a_{k+1}(\phi - \phi')^T B^{-1}(\phi + \phi') \right)\bigg | \xi_J, \xi'_J\right]\right|  \\
&\le\mb{E}_{\xi_J,\xi_J'}\bigg|\mb{E}_{\xi_I}\bigg[\exp\bigg(\frac{2\pi i}{p}(2a_{k+1}(\xi_J-\xi_J')^TB^{-1}\phi'')\bigg)\bigg|\xi_J,\xi_J'\bigg]\bigg|,
\end{align*}
where we have abused notation by using $\xi_J - \xi'_J$ to denote the extension of this vector to $[t-k]$ with the coordinates in $[t-k]\setminus J$ equal to $0$ and where $\phi''$ denotes the $(t-k)$-dimensional vector which coincides with $\xi$ (and hence $\xi'$) on $I \cap [t-k]$ and has remaining coordinates $0$. 

We now consider two cases. Note that $\mb{P}[\xi_J = \xi_J'] = 2^{-|J|}$, which is of size $\exp(-\Omega(t/\log p))$. Otherwise, $\bs{w} = B^{-1}(\xi_J-\xi'_{J})$ is a linear combination of at most $t/(64\log p)$ different columns of $B^{-1}$ and hence is orthogonal to at least $(t-k) -t/(64\log p) \geq t/(32\log{p})$ different rows of $B$. Hence, if we extend $\bs{w}$ to a $t$-dimensional vector by padding it with $0$s, then the resulting vector is a non-zero vector which is orthogonal to at least $t-t/32(\log{p})$ rows of $A_t$. Since $A_{t}$ is assumed to satisfy $\mc{E}_{U,t}$, it follows that this vector has at least $t/(2m)$ disjoint sets $S_1,\dots, S_{t/2m}$ of size $m = C_{\ref{prop:structure-many}}\log p$ such that $\on{disc}_{\mb{F}_p}(\bs{v}|_{S_i})\le C_{\ref{prop:structure-many}}p^{-2}$ for all $i \in [2m]$. Furthermore $\mc{E}_{U,t}$ guarantees that we can take these sets disjoint from the set $J\cup\{t-k+1,\ldots,t\}$ which has size at most $t/(32\log{p})$. In other words, the sets $S_1,\dots, S_{t/2m}$ are contained in $I \cap [t-k]$, which is the support of the random entries of $\phi''$. Thus, similar to Step 1, we deduce for realizations of $\xi_J,\xi_J'$ such that $\xi_{J} - \xi'_{J} \neq 0$,
\[\bigg|\mb{E}_{\xi_I}\bigg[\exp\bigg(\frac{2\pi i}{p}(2a_{k+1}(\xi_J-\xi_J')^TB^{-1}\phi'')\bigg)\bigg|\xi_J,\xi_J'\bigg]\bigg|\le(C_{\ref{prop:structure-many}}p^{-1})^{t/(2m)} = \exp(-\Omega(t))\]
for $p > C_{\ref{prop:structure-many}}$. For $p\le C_{\ref{prop:structure-many}}$, we use $\mc{E}_{S,t}$ to deduce that $\bs{w}$ has support size at least $t/(16\log p)$. Therefore, $\bs{w}$ has at least $t/(32\log p)$ on the support of the random entries of $\phi''$,  so that the result again follows as in Step 1. 
\end{proof}

\bibliographystyle{amsplain0.bst}
\bibliography{main.bib}
\appendix

\section{Proof of \texorpdfstring{\cref{thm:main}}{Theorem 1.1}}
\label{sec:nt}
In this section, we show how the local singularity statement \cref{thm:sharp-probability} implies the global irreducibility statement \cref{thm:main}. We use an approach pioneered by Breuillard and Varj\'u \cite{BV19} for random polynomials and used subsequently by Eberhard \cite{Ebe20} for characteristic polynomials of i.i.d.~matrices. The proof is nearly identical to that given in \cite[Section~3]{Ebe20} modulo the input of \cref{thm:sharp-probability}.

We first define some notation. Let $\Omega\subseteq\mb{C}$ be the set of roots of $\varphi$ and $G$ be its Galois group. Let $\Lambda_\varphi(p)$ be the number of roots of $\varphi$ in $\mb{F}_p$, without multiplicity. A number field $K$ has discriminant $\Delta_K$, while a polynomial $\varphi$ has discriminant $\Delta_\varphi$. Given $q\in\mb{Z}$, $P_K(q)$ is the number of prime ideals of $K$ of norm $q$.

\begin{proposition}[{\cite[Proposition~16]{BV19}}]\label{prop:ANT}
If $\varphi\in\mb{Z}[x]$ and $\wt{\varphi}$ is its squarefree part, and if $p\nmid\Delta_{\wt{\varphi}}$ then
\[R_\varphi(p) = \sum_{\omega\in\Omega/G}P_{\mb{Q}(\omega)}(p).\]
\end{proposition}

Now let
\[w_X(t) = 2\exp(-X)\mbm{1}_{(X-\log 2,X]}(t)\cdot t\]
be a weighting function.
\begin{proposition}[{\cite[Proposition~9]{BV19}}]\label{prop:ERH}
If the Riemann hypothesis holds for $K$ then
\[\sum_pP_K(p)w_X(\log p) = 1 + O(X^2\exp(-X/2)\log\Delta_K).\]
\end{proposition}
\begin{proposition}[{\cite[Proposition~3.5]{Ebe20}}]
Let $\varphi$ be the characteristic polynomial of a matrix $M$ with integer entries bounded in magnitude by $H$. If the Riemann hypothesis holds for $\mb{Q}(\omega)$ for all roots $\omega$ of $\varphi$, then
\[\sum_pR_\varphi(p)w_X(p) = |\Omega/G| + O(n^3X^2\exp(-X/2)\log(Hn)).\]
\end{proposition}

Finally, we state the following bound on the probability for a  symmetric Rademacher matrix to have simple spectrum.
\begin{proposition}\label{prop:simple}
The $n\times n$ random symmetric Rademacher matrix $M_n$ has simple spectrum with probability $1-\exp(-\Omega(n^{1/2}(\log n)^{1/4}))$.
\end{proposition}
\begin{remark}
A bound of the form $1 - \exp(-\Omega(n^{c})))$ for some small constant $c > 0$ is the content of \cite[Corollary~2.3]{nguyen2017random}. The improved bound stated here follows by replacing the application of the results of \cite{Ver14} in \cite{nguyen2017random} by the substantially sharper arithmetic structure estimates \cite[Theorem~4.8;Lemma~4.5]{JSS20a} appearing in work of the last three authors \cite{JSS20a}. We omit the standard details. \end{remark}

We are ready to prove the result.
\begin{proof}[Proof of \cref{thm:main} given \cref{thm:sharp-probability}]
Given a prime $p > 2$ and $\lambda\in\mb{F}_p$, let $\mc{E}_{p,\lambda}$ be the event that the characteristic polynomial $\varphi$ of our random symmetric matrix $M_n$ has $\lambda$ as a root over $\mb{F}_p$. By \cref{thm:sharp-probability} applied to $M_n-\lambda I_n$, we see that if $2 < p\le\exp(c_{\ref{thm:sharp-probability}}n^{1/4}))$, then
\[\mb{P}[\mc{E}_{p,\lambda}] = \frac{1+O(1/p)}{p}.\]
Summing over $\lambda$ yields
\[\mb{E}[R_\varphi(p)] = 1 + O(1/p).\]
Thus for $2\le X\le c_{\ref{thm:sharp-probability}}n^{1/4}$ we have
\begin{align*}
\mb{E}\sum_pR_\varphi(p)w_X(\log p) &= \sum_p(1 + O(1/p))w_X(\log p)\\
&= (1 + O(\exp(-X/2)))\sum_pw_X(\log p)\\
&= 1 + O(\exp(-X/2)) + O(X^2\exp(-X/2)).
\end{align*}
The second line is by \cref{prop:ANT} applied to $K = \mb{Q}$, or just the prime number theorem with Riemann error term. Applying \cref{prop:ERH} under ERH we obtain
\[\mb{E}|\Omega/G| + O(n^3X^2\exp(-X/2)\log n) = 1 + O(X^2\exp(-X/2))\]
hence
\[\mb{E}|\Omega/G| = 1 + O(n^3X^2\exp(-X/2)\log n).\]
Choosing $X = c_{\ref{thm:sharp-probability}}n^{1/4}$ at the top of its range, we deduce
\[\mb{E}|\Omega/G| = 1 + O(\exp(c n^{-1/4})).\]
Thus
\[\mb{P}[|\Omega/G| > 1] = O(\exp(c n^{-1/4})).\]
Furthermore, $|\Omega/G| = 1$ means that $\varphi$ is a perfect power of an irreducible polynomial.

Now to rule out the case of perfect powers and complete the proof, it suffices to show that the random symmetric matrix $M_n$ has simple spectrum with very high probability, say at least $1-\exp(-\Omega(\sqrt{n}))$. This is the content of \cref{prop:simple}.
\end{proof}

\section{Proof of \texorpdfstring{\cref{prop:structure-1}}{Proposition 2.4}}
\label{sec:app-proof-structure-1}
We require a version of Halasz's inequality as well as a `counting inverse Littlewood-Offord theorem' tailored to it. This was developed in work of the first two authors along with Luh and Samotij \cite{FJLS2018}. For the sake of simplicity, we will use the statements \cite[Theorem~3.8,~3.9]{Jai19}.
\begin{definition}
Let $\bs{a}\in\mb{F}_p^n$ and $k\in\mb{N}$. We define $R_k^{\ast}(\bs{a})$ to be the number of solutions to 
\[\pm a_{i_1}\pm a_2\pm \ldots\pm a_{i_{2k}}\equiv 0 \mod p\]
with $|\{i_1,\ldots,i_{2k}\}|>1.01k$.
\end{definition}
\begin{theorem}[{\cite[Theorem~3.8]{Jai19},~c.f.~\cite[Theorem~1.4]{FJLS2018}}]\label{thm:halasz-fp}
Given an odd prime $p$, integer $n$, and vector $\bs{a}=(a_1,\ldots, a_n) \in\mb{F}_p^{n}\setminus\{\bs{0}\}$, suppose that an integer $0\le k \le n/2$ and positive real $L$ satisfy $30L \le |\on{supp}(\bs{a})|$ and $80kL \le n$. Then
\[\rho_{\mb{F}_p}(\bs{a})\le\frac{1}{p}+C_{\ref{thm:halasz-fp}}\frac{R_k^\ast(\bs{a}) + (40k^{0.99}n^{1.01})^k}{2^{2k} n^{2k} L^{1/2}} + e^{-L}. \]
\end{theorem}

\begin{theorem}[{\cite[Theorem~3.9]{Jai19}}]
\label{thm:counting}
Let $p$ be a prime and let $k, s_1,s_2,d\in [n], t\in [1,p]$ be such that $s_1 \le s_2$. 
Let 
\begin{small}
\[\mr{Bad}^{d}_{k,s_1,s_2,\ge t}(n) = \Big\{\bs{a} \in \mb{F}_{p}^{n}\colon |\on{supp}(\bs{a})|=d \text{ and } \forall \bs{b}\subset \bs{a}|_{\on{supp}(\bs{a})} \text{ s.t. } s_2\ge|\bs{b}|\ge s_1 : R_k^\ast(\bs{b})\ge t\cdot \frac{2^{2k}\cdot |\bs{b}|^{2k}}{p}\Big\}.\]
\end{small}
Then, 
\[|\mr{Bad}^{d}_{k,s_1,s_2,\ge t}(n)|\le\binom{n}{d}p^{d+s_2}(0.01t)^{-d+\frac{s_1}{s_2}d}.\] 
\end{theorem}

We now are in position to prove \cref{prop:structure-1}. The proof given is essentially identical to that in \cite{FJ19}. However, we need to be more careful regarding the level of unstructuredness obtained in the argument.

\begin{proof}[Proof of \cref{prop:structure-1}]
Let $r = n/\log p$. Let $k = n^{1/4}$, $s_1 = \sqrt{n\log n}$, $s_2 = n^{3/4}$, and choose some $n/(16\log p)\le d\le n$ (so that $s_2\le d$ in particular). First use \cref{lem:not-sparse} to rule out $d$-sparse vectors $\bs{v}$. Next, let $L = n^{1/4}$. Consider some $\sqrt{L}\le t\le p$, if it exists.

Consider a fixed $\bs{v}\in\mr{Bad}_{k,s_1,s_2,\ge t}^d\setminus\mr{Bad}_{k,s_1,s_2,\ge 2t}^d$ and a fixed choice of $n-r$ rows of $M_n-\lambda I_n$. We wish to bound the probability that $\bs{v}$ is orthogonal to all those rows. By definition, there is a set $S\subseteq\on{supp}(\bs{v})$ of size in $[s_1,s_2]$ such that
\[R_k^\ast(\bs{v}|_S) < 2t\cdot\frac{2^{2k}|S|^{2k}}{p}.\]
Since $\bs{v}$ is orthogonal to all of the given $n-r$ rows, we expose them one-by-one (with any rows in $S$ coming last). The first at least $n-r-s_2$ rows are such that, conditioned on the prior revelations, the random dot product with $\bs{v}$ is zero with probability at most $\rho_{\mb{F}_p}(\bs{v}|_S)$. Furthermore, by \cref{thm:halasz-fp} and the given conditions, which guarantee $30L\le s_1\le|\on{supp}(\bs{v}|_S)|$ and $80kL\le s_1\le|S|$, we have
\begin{align*}
\rho_{\mb{F}_p}(\bs{v}|_S)&\le\frac{1}{p}+C_{\ref{thm:halasz-fp}}\frac{R_k^\ast(\bs{v}|_S) + (40k^{0.99}|S|^{1.01})^k}{2^{2k} |S|^{2k} L^{1/2}} + e^{-L}\\
&\le\frac{1}{p} + \frac{2C_{\ref{thm:halasz-fp}}t}{p\sqrt{L}} + \frac{10^k C_{\ref{thm:halasz-fp}}}{L^{1/2}}\bigg(\frac{k}{|S|}\bigg)^{0.99k} + e^{-L}\\
&\le\frac{Ct}{p\sqrt{L}}
\end{align*}
for all sufficiently large $n$. 
Multiplying over all the rows, and taking a union bound over the possible choices of $\bs{v}$ (\cref{thm:counting}) and collection of $n-r$ rows, we have a bound of
\begin{align*}
\binom{n}{d}p^{d+s_2}(0.01t)^{-d+\frac{s_1}{s_2}d}&\cdot\binom{n}{r}\cdot\bigg(\frac{Ct}{p\sqrt{L}}\bigg)^{n-r-s_2}\\
&\le (C')^np^dt^{-d+\frac{s_1}{s_2}d}\bigg(\frac{t}{p\sqrt{L}}\bigg)^n\\
&\le\exp(C''n\sqrt{\log n})(t/p)^{n-d}L^{-n/2}\le\exp(-n(\log n)/9)
\end{align*}
for sufficiently large $n$. Union bounding over all possible choices of $d$ and a dyadic partition of $t$ shows that there is an appropriately small chance of having such a vector orthogonal to $n-r$ rows for any $t\ge\sqrt{M}$.

The remaining vectors $\bs{v} \in \mb{F}_p^n$ with $|\on{supp}(\bs{v})| \geq n/(16\log{p})$ all have some subset $S\subseteq\on{supp}(\bs{v})$ such that
\[R_k^\ast(\bs{v}|_S)\le\sqrt{L}\cdot\frac{2^{2k}|S|^{2k}}{p},\]
and applying \cref{thm:halasz-fp} once again finishes.
\end{proof}

\end{document}